\documentclass[letter]{amsart}
\usepackage{amssymb, mathtools, enumitem, hyperref, cleveref}

\newtheorem{theorem}{Theorem}
\numberwithin{theorem}{section}
\newtheorem{proposition}[theorem]{Proposition}
\newtheorem{lemma}[theorem]{Lemma}
\newtheorem{corollary}[theorem]{Corollary}

\theoremstyle{definition}

\newtheorem{example}[theorem]{Example}
\newtheorem{remark}[theorem]{Remark}
\crefname{question}{question}{questions}

\newcommand{\N}{\mathbb{N}}
\newcommand{\Z}{\mathbb{Z}}
\newcommand{\Q}{\mathbb{Q}}
\newcommand{\m}{\mathfrak{m}}
\newcommand{\PP}{\mathbb{P}}
\newcommand{\F}{\mathbb{F}}
\DeclareMathOperator{\mono}{mono}
\DeclareMathOperator{\Spec}{Spec}
\DeclareMathOperator{\codim}{codim}
\DeclareMathOperator{\Mono}{Mono}
\DeclareMathOperator{\Min}{Min}
\DeclareMathOperator{\Ass}{Ass}
\DeclareMathOperator{\ch}{char}
\DeclareMathOperator{\reg}{reg}
\DeclareMathOperator{\soc}{soc}

\newcommand{\arxiv}[1]{\href{http://arxiv.org/abs/#1}{{\tt arXiv:#1}}}

\begin{document}
\title{Mono: an algebraic study of torus closures}
\author{Justin Chen}
\address{Department of Mathematics, University of California, Berkeley,
California, 94720 U.S.A}
\email{jchen@math.berkeley.edu}


\begin{abstract}
Given an ideal $I$ in a polynomial ring, we consider the largest monomial subideal contained in $I$, denoted $\mono(I)$. We study mono as an interesting operation in its own right, guided by questions that arise from comparing the Betti tables of $I$ and $\mono(I)$. Many examples are given throughout to illustrate the phenomena that can occur.
\end{abstract}

\maketitle

Let $R = k[x_1, \ldots, x_n]$ be a polynomial ring over a field $k$ in $n$ variables. For any ideal $I \subseteq R$, let $\mono(I)$ denote the largest monomial subideal of $I$, i.e. the ideal generated by all monomials contained in $I$. Geometrically, $\mono(I)$ defines the smallest torus-invariant subscheme containing $V(I) \subseteq \Spec R$ (the so-called \textit{torus-closure} of $V(I)$). 

The concept of mono has been relatively unexplored, despite the naturality of the definition. The existing work in the literature concerning mono has been essentially algorithmic and/or computational. For convenience, we summarize this in the following two theorems: 

\begin{theorem}[\cite{SST}, Algorithm 4.4.2]
Let $I = (f_1, \ldots, f_r)$. Fix new variables $y_1, \ldots, y_n$, and let $\widetilde{f_i} := f_i(\frac{x_1}{y_1}, \ldots, \frac{x_n}{y_n}) \cdot \prod_{i=1}^n y_i^{\deg_{x_i}(f)}$ be the multi-homogenization of $f_i$ with respect to $\underline{y}$. Let $>$ be an elimination term order on $k[\underline{x}, \underline{y}]$ satisfying $y_i > x_j$ for all $i, j$. If $\mathcal{G}$ is a reduced Gr\"obner basis for $(\widetilde{f_1}, \ldots, \widetilde{f_r}) : (\prod_{i=1}^n y_i)^\infty$ with respect to $>$, then the monomials in $\mathcal{G}$ generate $\mono(I)$.
\end{theorem}

Cf. also \cite{M} for a generalization computing the largest $A$-graded subideal of an ideal, for an integer matrix $A$ ($\mono$ being the special case when $A$ is the identity matrix). The next theorem gives an alternate description of $\mono$ for a particular class of ideals, involving the dual concept of $\Mono$, which is the smallest monomial ideal containing a given ideal (notice that $\Mono(I)$ is very simple to compute, being generated by all terms appearing in a generating set of $I$).

\begin{theorem}[\cite{PUV}, Lemma 3.2]
Let $I$ be an unmixed ideal, and suppose there exists a regular sequence $\underline{\beta} \subseteq I$ consisting of $\codim I$ monomials. Then $\mono(I) = (\underline{\beta}) : \Mono( (\underline{\beta}) : I)$.
\end{theorem}

However, it appears that no systematic study of $\mono$ as a operation on ideals has yet been made. It is the goal of this note to provide first steps in this direction; in particular exploring the relationship between $I$ and $\mono(I)$. By way of understanding $\mono$ as an algebraic process, we consider the following questions:

\begin{enumerate}

\item \label[question]{Q1} When is $\mono(I) = 0$, or prime, or primary, or radical? 

\item \label[question]{Q2} To what extent does taking $\mono$ depend on the ground field?

\item \label[question]{Q3} Which invariants are preseved by taking $\mono$? For instance, do $I$ and $\mono(I)$ have the same (Castelnuovo-Mumford) regularity?

\item \label[question]{Q4} How do the Betti tables of $I$ and $\mono(I)$ compare? Do they have the same shape?

\item \label[question]{Q5} To what extent is $\mono$ non-unique? E.g. which monomial ideals arise as $\mono$ of a non-monomial ideal?

\item \label[question]{Q6} What properties of $I$ are preseved by $\mono(I)$, and conversely, what properties of $I$ are reflected by $\mono(I)$?

\newcounter{enumTemp}
\setcounter{enumTemp}{\theenumi}

\end{enumerate}

\section{Basic properties}

We first give some basic properties of $\mono$, which describe how $\mono$ interacts with various algebraic operations. As above, $R$ denotes a polynomial ring $k[x_1, \ldots, x_n]$, and $\m = (x_1, \ldots, x_n)$ denotes the homogeneous maximal ideal of $R$.

\begin{proposition} \label{basicProp}

Let $I$ be an $R$-ideal.

\begin{enumerate}

\item $\mono$ is decreasing, inclusion-preserving, and idempotent.

\item $\mono$ commutes with radicals, i.e. $\mono(\sqrt{I}) = \sqrt{\mono(I)}$.

\item $\mono$ commutes with intersections, i.e. $\mono(\bigcap I_i) = \bigcap \mono(I_i)$ for any ideals $I_i$.

\item $\mono(I_1) \mono(I_2) \subseteq \mono(I_1 I_2) \subseteq \mono(I_1) \cap \mono(I_2)$.

\end{enumerate}
\end{proposition}

\begin{proof}
\begin{enumerate}

\item Each property -- $\mono(I) \subseteq I$, $I_1 \subseteq I_2 \implies \mono(I_1) \subseteq \mono(I_2)$, $\mono(\mono(I)) = \mono(I)$ -- is clear from the definition.

\item If $u \in \sqrt{\mono(I)}$ is a monomial, say $u^m \in \mono(I)$, then $u \in \sqrt{I} \implies u \in \mono(\sqrt{I})$. Conversely, if $u \in \mono(\sqrt{I})$ is monomial, then $u \in \sqrt{I}$, say $u^m \in I$ and hence $u^m \in \mono(I) \implies u \in \sqrt{\mono(I)}$.

\item $\bigcap I_i \subseteq I_i \implies \mono(\bigcap I_i) \subseteq \mono(I_i)$, hence $\mono(\bigcap I_i) \subseteq \bigcap \mono(I_i)$. On the other hand, an arbitrary intersection of monomial ideals is monomial, 
% use definition: M is monomial iff whenever f \in M, each term of f is in M
and $\bigcap \mono(I_i) \subseteq \bigcap I_i$, hence $\bigcap \mono(I_i) \subseteq \mono(\bigcap I_i)$. 

\item $\mono(I_1) \mono(I_2) \subseteq I_1 I_2$, and a product of two monomial ideals is monomial, hence $\mono(I_1) \mono(I_2) \subseteq \mono(I_1 I_2)$. The second containment follows from applying (1) and (3) to the containment $I_1 I_2 \subseteq I_1 \cap I_2$. \qedhere

\end{enumerate}
\end{proof}

Next, we consider how prime and primary ideals behave under taking $\mono$:

\begin{proposition} \label{primeProp}
Let $I$ be an $R$-ideal. 
\begin{enumerate}

\item If $I$ is prime resp. primary, then so is $\mono(I)$.

\item $\Ass(R/ \mono(I)) \subseteq \{ \mono(P) \mid P \in \Ass(R/I) \}$.

\item $\mono(I)$ is prime iff $\mono(I) = \mono(P)$ for some minimal prime $P$ of $I$. In particular, $\mono(I) = 0$ iff $\mono(P) = 0$ for some $P \in \Min(I)$.

\end{enumerate}
\end{proposition}

\begin{proof}
\begin{enumerate}

\item To check that $\mono(I)$ is prime (resp. primary), it suffices to check that if $u, v$ are monomials with $uv \in \mono(I)$, then $u \in \mono(I)$ or $v \in \mono(I)$ (resp. $v^m \in \mono(I)$ for some $n$). But this holds, as $I$ is prime (resp. primary) and $u, v$ are monomials. 

\item If $I = Q_1 \cap \ldots \cap Q_r$ is a minimal primary decomposition of $I$, so that $\Ass(R/I) = \{ \sqrt{Q_i} \}$, then $\mono(I) = \mono(Q_1) \cap \ldots \cap \mono(Q_r)$ is a primary decomposition of $\mono(I)$, so every associated prime of $\mono(I)$ is of the form $\sqrt{\mono(Q_i)} = \mono(\sqrt{Q_i})$ for some $i$.

\item $\displaystyle \mono(I) = \sqrt{\mono(I)} = \mono(\sqrt{I}) = \bigcap_{P \in \Min(I)} \mono(P)$. Since $\mono(I)$ is prime and the intersection is finite, $\mono(I)$ must equal one of the terms in the intersection. The converse follows from (1). \qedhere

\end{enumerate}
\end{proof}

We now examine the sharpness of various statements in \Cref{basicProp,primeProp}:

\begin{example} \label{twoQuadricsEx}
Let $R = k[x,y]$, where $k$ is an infinite field, and let $I$ be an ideal generated by 2 random quadrics. Then $R/I$ is an Artinian complete intersection of regularity $2$, so $\m^3 \subseteq I$. By genericity, $I$ does not contain any monomials in degrees $\le \reg(R/I)$, so $\mono(I) = \m^3$. On the other hand, $I^2$ is a 3-generated perfect ideal of grade 2, so the Hilbert-Burch resolution of $R/I^2$ shows that $\reg R/I^2 = 4$, hence $\m^5 \subseteq I^2 \subseteq \m^4$ as $I^2$ is generated by quartics (in fact, $\mono(I^2) = \m^5$). Thus for such $I = I_1 = I_2$, both containments in \Cref{basicProp}(4) are strict.
\end{example}

Similar to \Cref{basicProp}(4), the containment in \Cref{primeProp}(2) is also strict in general: take e.g. $I = I' \cap \m^N$ where $\mono(I') = 0$ and $N > 0$ is such that $I' \not \subseteq \m^N$. However, combining these two statements yields:

\begin{corollary}
Let $I$ be an $R$-ideal.

\begin{enumerate}
\item Nonzerodivisors on $R/I$ are also nonzerodivisors on $R/\mono(I)$.

\item Let $u \in R$ be a monomial that is a nonzerodivisor on $R/I$. Then $\mono((u)I) \\ = (u) \mono(I)$.
\end{enumerate}
\end{corollary}

\begin{proof}
\begin{enumerate}
\item $\displaystyle \bigcup_{P \in \Ass(R/\mono(I))} P \subseteq \bigcup_{P \in \Ass(R/I)} \mono(P) \subseteq \bigcup_{P \in \Ass(R/I)} P$. 

\item This follows from (1) and \Cref{basicProp}(4). \qedhere
\end{enumerate}
\end{proof}

\begin{remark} \label{zeroRmk}
Since monomial prime ideals are generated by (sets of) variables, if $P \subseteq \m^2$ is a nondegenerate prime, then $\mono(P) = 0$. It thus follows from \Cref{primeProp}(3) that ``most" ideals $I$ satisfy $\mono(I) = 0$: namely, this is always the case unless each component of $V(I)$ is contained in some coordinate hyperplane in $\mathbb{A}^n = \Spec R$.
% (in particular, $\mono(I)$ is highly sensitive to the choice of coordinates on $\Spec R$) 
The case that $\mono(I)$ is prime is analogous: if $\mono(I) = (x_{i_1}, \ldots, x_{i_r})$ for some $\{i_1, \ldots, i_r\} \subseteq \{1, \ldots, n \}$, then $V(I)$ becomes nondegenerate upon restriction to the coordinate subspace $V(x_{i_1}, \ldots, x_{i_r}) \cong \mathbb{A}^r$ (i.e. $\mono(\overline{I}) = \overline{0}$ where $\overline{( \cdot )}$ denotes passage to the quotient $R/(x_{i_1}, \ldots, x_{i_r})$.)
\end{remark}

In contrast to the simple picture when $\mono(I)$ is prime, the case where $\mono(I)$ is primary is much more interesting, due to nonreducedness issues. The foremost instance of this case is when $\mono(I)$ is $\m$-primary, i.e. $\mono(I)$ is Artinian. A first indication that this case is interesting is that under this assumption, $\mono(I)$ is guaranteed not to be $0$. For this and other reasons soon to appear, we will henceforth deal primarily with this case -- the reader should assume from now on that $I$ is an Artinian ideal.

\section{Dependence on scalars}

We now briefly turn to \Cref{Q2}: to what extent does taking $\mono$ depend on the ground field $k$? To make sense of this, let $S = \Z[x_1, \ldots, x_n]$ be a polynomial ring over $\Z$. Then for any field $k$, the universal map $\Z \to k$ induces a ring map $S \to S_k := S \otimes_\Z k = k[x_1, \ldots, x_n]$. Given an ideal $I \subseteq S$, one can consider the extended ideal $IS_k$. The question is then: as the field $k$ varies, how does $\mono(IS_k)$ change?

It is easy to see that if $k_1$ and $k_2$ have the same characteristic, then $\mono(IS_{k_1})$ and $\mono(IS_{k_2})$ have identical minimal generating sets. Thus it suffices to consider prime fields $\Q$ and $\F_p$, for $p \in \Z$ prime. Another moment's thought shows that $\mono$ can certainly change in passing between different characteristics; e.g. if all but one of the coefficients of some generator of $I$ is divisible by a prime $p$. However, even excluding simple examples like this, by requiring that the generators of $I$ all have unit coefficients, $\mono$ still exhibits dependence on characteristic. We illustrate this with a few examples:

\begin{example} 

Let $S = \Z[x,y,z]$ be a polynomial ring in $3$ variables.

(1) Set $I := (x^3, y^3, z^3, xy(x + y + z))$. Then $xyz^2 \in \mono(IS_k)$ iff $\ch k = 2$ (consider $xy(x + y + z)^2 \in I$). Notice that $I$ is equi-generated, i.e. all minimal generators of $I$ have the same degree.

(2) For a prime $p \in \Z$, set $I_p := (x^p, y^p, x + y + z)$. Then $z^p \in \mono(I_pS_k)$ iff $\ch k = p$ (consider $(x + y + z)^p \in I_p$). If the presence of the linear form is objectionable, one may increase the degrees, e.g. $(x^{2p}, y^{2p}, x^2 + y^2 + z^2)$. 
\end{example}

From these examples we see that $\mono$ is highly sensitive to characteristic in general. However, this is not the whole story: cf. \Cref{colonCharInd} for one situation where taking $\mono$ is independent of characteristic.

\section{Betti tables}

We now consider how invariants of $I$ behave when passing to $\mono(I)$. As mentioned in \Cref{zeroRmk}, although $I$ and $\mono(I)$ are typically quite different, for Artinian graded ideals there is a much closer relationship:

\begin{proposition} \label{regMono}
Let $I$ be a graded $R$-ideal. Then $I$ is Artinian iff $\mono(I)$ is Artinian.
In this case, $\reg(R/I) = \reg(R/\mono(I))$.
\end{proposition}

\begin{proof}
Since $I \subseteq \m$ is graded, $I$ is Artinian iff $\m^s \subseteq I$ for some $s > 0$. This occurs iff $\m^s \subseteq \mono(I)$ for some $s > 0$ iff $\mono(I)$ is Artinian. 

Next, recall that if $M = \bigoplus M_i$ is Artinian graded, then the regularity of $M$ is $\reg M = \max \{i \mid M_i \ne 0 \}$.
The inclusion $\mono(I) \subseteq I$ induces a (graded) surjection $R/\mono(I) \twoheadrightarrow R/I$,
which shows that $\reg(R/\mono(I)) \ge \reg(R/I)$. 
Now if $u \in R$ is a standard monomial of $\mono(I)$ of top degree ($= \reg(R/\mono(I))$), then $u \not \in 
\mono(I) \implies u \not \in I$, hence $\reg(R/I) \ge \deg u = \reg(R/\mono(I))$.
\end{proof}

A restatement of \Cref{regMono} is that for any Artinian graded ideal $I$, the graded Betti tables of 
$I$ and $\mono(I)$ have the same number of rows and columns (since any Artinian ideal has projective dimension 
$n = \dim R$). However, it is not true (even in the Artinian case) that the Betti tables of $I$ and $\mono(I)$ have the 
same shape (= (non)zero pattern) -- e.g. take an ideal $I'$ with $\mono(I') = 0$, and 
consider $I := I' + \m^N$ for $N \gg 0$. Despite this, there is one positive result in this direction:

\begin{proposition} \label{BettiTopNonVanishing}
Let $I$ be an Artinian graded $R$-ideal. Then $\beta_{n,j}(R/\mono(I)) \ne 0 \implies \beta_{n,j}(R/I) \ne 0$, 
for any $j$.
\end{proposition}

\begin{proof}
Notice that $\beta_{n,j}(R/I) \ne 0$ iff the socle of $R/I$ contains a nonzero form of degree $j$. 
Let $m \in R$ be a monomial with $\overline{0} \ne \overline{m} \in \soc(R/\mono(I))$ and $\deg m = j$. Then 
$m \in (\mono(I) :_R \m) \setminus \mono(I)$, hence $m \in (I :_R \m) \setminus I$ as well,
i.e. $\overline{0} \ne \overline{m} \in \soc(R/I)$.
\end{proof}

\begin{corollary} \label{levelProp}
Let $I$ be an Artinian graded level $R$-ideal (i.e. $\soc(R/I)$ is nonzero in only one degree). Then 
$\mono(I)$ is also level, with the same socle degree as $I$.
\end{corollary}

\begin{proof}
Follows immediately from \Cref{BettiTopNonVanishing}.
\end{proof}

We illustrate these statements with some examples of how the Betti tables of $R/I$ and $R/\mono(I)$ can differ:
\vspace{-0.1cm}

\begin{example}
Let $R = k[x,y,z,w]$, $J = I_2
\begin{pmatrix}
x & y^2 & yw & z \\
y & xz & z^2 & w
\end{pmatrix}$ the ideal of the rational quartic curve in $\PP^3$, and $I = J + (x^2, y^4, z^4, w^4)$. 
Then the Betti tables of $R/I$ and $R/\mono(I)$ respectively in Macaulay2 format are:
\begin{verbatim}
              0 1  2  3 4               0  1  2  3 4
       total: 1 7 15 13 4        total: 1 11 28 26 8
           0: 1 .  .  . .            0: 1  .  .  . .
           1: . 2  .  . .            1: .  1  .  . .
           2: . 3  5  1 .            2: .  2  1  . .
           3: . 2  5  4 1            3: .  6 10  5 1
           4: . .  4  5 1            4: .  2 14 14 3
           5: . .  1  3 2            5: .  .  3  7 4
\end{verbatim}
Notice that $\beta_{1,5}(R/\mono(I)) = 2 \ne 0 = \beta_{1,5}(R/I)$, and likewise $\beta_{3,5}(R/I) = 1 \ne 0 = \beta_{3,5}(R/\mono(I))$.
In addition, $\beta_{1,2}, \beta_{1,3}, \beta_{2,4}$ for $R/\mono(I)$ are all strictly smaller than their counterpart for $R/I$ (and still nonzero).
\end{example}

\begin{example}
Let $R = k[x,y,z]$, $\ell \in R_1$ a general linear form (e.g. $\ell = x + y + z$), and $I = (x \ell, y \ell, z \ell) + (x,y,z)^3$. Then $\mono(I) = (x,y,z)^3$, and the Betti tables of $R/I$ and $R/\mono(I)$ respectively are:
\begin{verbatim}
                0 1  2 3                  0  1  2 3
         total: 1 7 10 4           total: 1 10 15 6
             0: 1 .  . .               0: 1  .  . .
             1: . 3  3 1               1: .  .  . .
             2: . 4  7 3               2: . 10 15 6
\end{verbatim}
Here $\mono(I)$ is level, but $I$ is not. This shows that the converse to \Cref{levelProp} is not true in general.
\end{example}

\begin{example}
We revisit \Cref{twoQuadricsEx}. Since $\mono(I)$ is a power of the maximal ideal, $R/\mono(I)$ has a
linear resolution, whereas $R/I$ has a Koszul resolution (with no linear forms), so in view of \Cref{BettiTopNonVanishing}, the 
Betti tables of $R/I$ and $R/\mono(I)$ have as disjoint shapes as possible. Thus no analogue of \Cref{BettiTopNonVanishing}
can hold for $\beta_i$, $i < n$, in general.
\end{example}

From the examples above, one can see that the earlier propositions on Betti tables are fairly sharp. Another interesting pattern observed above is that even when the ideal-theoretic description of $\mono(I)$ became simpler than that of $I$, the Betti table often grew worse (e.g. had larger numbers on the whole). This leads to some natural refinements of \Cref{Q4,Q6}:
\vspace{0.2cm}
\begin{enumerate}
\setcounter{enumi}{\theenumTemp}
\item \label[question]{Q7} Are the total Betti numbers of $\mono(I)$ always at least those of $I$?

\item \label[question]{Q8} Does $\mono(I)$ Gorenstein imply $I$ Gorenstein?
\end{enumerate}
\vspace{0.2cm}
\noindent
Notice that the truth of \Cref{Q7} implies the truth of \Cref{Q8}. As it turns out, the answer to these will follow from the answer to \Cref{Q5}.

\section{Uniqueness and the Gorenstein property}

\begin{lemma} \label{equalColonLemma}
Let $M$ be a monomial ideal, and $u_1 \ne u_2$ standard monomials of $M$. Then $\mono(M + (u_1 + u_2)) = M$ iff $M : u_1 = M : u_2$.
\end{lemma}

\begin{proof}
$\Rightarrow$: By symmetry, it suffices to show that $M : u_1 \subseteq M : u_2$. Let $m$ be a monomial in $M : u_1$. Then $mu_2 = m(u_1 + u_2) - mu_1 \in \mono(M + (u_1 + u_2)) = M$, i.e. $m \in M : u_2$.

$\Leftarrow$: Passing to $R/M$, it suffices to show that $(\overline{u_1} + \overline{u_2})$ contains no monomials in $R/M$. Let $g \in R$ be such that $\overline{g}(\overline{u_1} + \overline{u_2}) \ne \overline{0} \in R/M$, and write $g = g_1 + \ldots + g_s$ as a sum of monomials. By assumption, $g_i u_1 \in M$ iff $g_i u_2 \in M$, so after removing some terms of $g$ we may assume there exists $g_i$ of top degree in $g$ such that $g_i u_1, g_i u_2 \not \in M$. But then $\overline{g_i} \overline{u_1}$ and $\overline{g_i} \overline{u_2}$ both appear as distinct terms in $\overline{g}(\overline{u_1} + \overline{u_2})$, so $\overline{g}(\overline{u_1} + \overline{u_2})$ is not a monomial in $R/M$.
\end{proof}

% Q: How to extend \Cref{equalColonLemma} to determine when mono(M + (f)) = M? Or more generally, how to describe all I such that mono(I) = M?

\begin{remark} \label{colonCharInd}
Since colons of monomial ideals are characteristic-independent, the second condition in \Cref{equalColonLemma} is independent of the ground field $k$. Thus if $I$ is an ideal defined over $\Z$ which is ``nearly" monomial (i.e. is generated by monomials and a single binomial), and $\mono(I)$ is as small as possible in one characteristic, then $\mono(I)$ is the same in all characteristics.
\end{remark}

\begin{remark} \label{decreaseSupportRmk}
For any polynomial $f \in R$, it is easy to see that 
\[
\mono(M + (f)) \supseteq M + \sum_{u \in \operatorname{terms}(f)} \mono(M : f - u)u
\]
If $f = u_1 + u_2$ is a binomial, then this simplifies to the statement that $\mono(M + (u_1 + u_2)) \supseteq M + (M : u_2)u_1 + (M : u_1)u_2$. However, equality need not hold: e.g. $M = (x^6, y^6, x^2y^4)$, $u_1 = x^2y$, $u_2 = xy^2$ (or even $M = (x^3, y^2)$, $u_1 = x$, $u_2 = y$ if one allows linear forms).
\end{remark}

\begin{theorem} \label{monoCharacterizationThm}
The following are equivalent for a monomial ideal $M$:

\begin{enumerate}[label={\upshape(\arabic*)}]
\item There exists a graded non-monomial ideal $I$ such that $\mono(I) = M$.
\item There exist $t \ge 2$ monomials $u_1, \ldots, u_t$ not contained in $M$ and of the same degree, such that $M : u_i = M : u_j$ for all $i, j$.
\item There exist monomials $u_1 \ne u_2$ with $u_1, u_2 \not \in M$, $\deg u_1 = \deg u_2$ and $M : u_1 = M : u_2$.
\end{enumerate}
\end{theorem}

\begin{proof}
(1) $\implies$ (2): Fix $f \in I \setminus M$ graded of minimal support size $t$ (so $t \ge 2$), and write $f = u_1 + \ldots + u_t$ where $u_i$ are standard monomials of $M$ of the same degree. Fix $1 \le i \le t$, and pick a monomial $m \in M : u_i$. Then $m(f - u_i) \in I$ has support size $< t$, so minimality of $t$ gives  $m(f - u_i) = \sum_{j \ne i} mu_j \in M$. Since $M$ is monomial, $mu_j \in M$ for each $j \ne i$, i.e. $m \in M : u_j$ for all $j$. By symmetry, $M : u_i = M : u_j$ for all $i, j$.

(2) $\implies$ (3): Clear.

(3) $\implies$ (1): Set $I := M + (u_1 + u_2)$, and apply \Cref{equalColonLemma}. Notice that $I$ is not monomial: if it were, then $u_1 + u_2 \in I \implies u_1 \in I \implies u_1 \in \mono(I) = M$, contradiction.
\end{proof}

\begin{corollary} \label{characterizationMonoGorenstein}
Let $I$ be an Artinian graded $R$-ideal. Then $\mono(I)$ is a complete intersection iff $\mono(I)$ is Gorenstein iff $I = \m^b  := (x_1^{b_1}, \ldots, x_n^{b_n})$ for some $b \in \N^n$.
\end{corollary}

\begin{proof}
Any Artinian Gorenstein monomial ideal is irreducible, hence is of the form $\m^b$, which is a complete intersection. By \Cref{monoCharacterizationThm}, it suffices to show that for $M := \m^b$, no distinct standard monomials of $M$ satisfy $M : u_1 = M : u_2$. To see this, note that since $u_1 \ne u_2$, there exists $j \in [n]$ such that $x_j$ appears to different powers $a_1 \ne a_2$ in $u_1$ and $u_2$, respectively. Taking $a_1 < a_2$ WLOG gives $x_j^{b_j - a_2} \in (M : u_2) \setminus (M : u_1)$.
\end{proof}

Combining the proofs above shows that an Artinian monomial ideal is not expressible as mono of \textit{any} non-monomial ideal iff it is Gorenstein:

\begin{corollary} \label{nongradedCharMono}
Let $M$ be an Artinian monomial ideal. Then there exists a non-monomial $R$-ideal $I$ with $\mono(I) = M$ iff $M$ is not Gorenstein (iff $M$ is not of the form $\m^b$ for $b \in \N^n$).
\end{corollary}

\begin{proof}
$\Rightarrow$: If $M = \mono(I)$ were Gorenstein, then by \Cref{characterizationMonoGorenstein} $I$ is necessarily of the form $\m^b$, contradicting the hypothesis that $I$ is non-monomial. 

$\Leftarrow$: Since $M$ is not Gorenstein, there exist monomials $u_1 \ne u_2$ in the socle of $R/M$. Then $u_1 \in M : \m \implies \m \subseteq M : u_1 \implies \m = M : u_1$, and similarly $\m = M : u_2$. By \Cref{equalColonLemma}, $I := M + (u_1 + u_2)$ is a non-monomial ideal with $\mono(I) = M$.
\end{proof}

As evidenced by \Cref{decreaseSupportRmk}, finding $\mono(M + (f))$ can be subtle, for arbitrary $f \in R$. There is one situation however which can be determined completely:

\begin{theorem} \label{socTerms}
Let $M$ be a monomial ideal, and let $u_1, \ldots, u_r$ be the socle monomials of $R/M$. Let $f_j := \sum_{i=1}^r a_{ij} u_i$, $1 \le j \le s$, be $k$-linear combinations of the $u_i$. Then $\mono(M + (f_1, \ldots, f_s)) = M$ iff no standard basis vector $e_i$ is in the column span of the matrix $(a_{ij})$ over $k$.
\end{theorem}

\begin{proof}
Let $v \in \mono(M + (f_1, \ldots, f_r))$ be a monomial. Pick $g_i \in R$ and $m \in M$ such that $v = m + \sum_{j=1}^s g_j f_j$. Write $g_j = b_j + g_j'$ where $b_j \in k$ and $g_j' \in \m$. Since $f_j \in \soc(R/M)$, this is the same as saying $\overline{v} = \sum_{j=1}^s b_j \overline{f_j}$ in $R/M$. Since $v$ is a monomial, it must appear as one of the terms in the sum, hence $v$ must be a socle monomial of $M$. Then $v = u_i$ for some $i$, so $v$ corresponds to a standard basis vector $e_i$, and then writing $v$ as a $k$-linear combination of $f_j$ is equivalent to writing $e_i$ as a $k$-linear combination of the columns of $(a_{ij})$.
\end{proof}

\begin{example}
Let $R = k[x,y,z]$, and set $M := (x^2,xy,xz,y^2,z^2)$. Then $R/M$ has a 2-dimensional socle $k \langle x, yz \rangle$, so $\mono(M + (x + yz)) = M$ by \Cref{equalColonLemma} or \Cref{socTerms}. However, the only standard monomials of $M$ of the same degrees are $x, y, z$, which have distinct colons into $M$. Thus by \Cref{monoCharacterizationThm} there is no \textit{graded} non-monomial $I$ with $\mono(I) = M$.

In general, even if there are $u_1, u_2$ of the same degree with $M : u_1 = M : u_2$, there may not be any such in top degree: e.g. $(x^2, y^2)\m + (z^3)$ is equi-generated with symmetric Hilbert function $1, 3, 6, 3, 1$, but is not level (hence not Gorenstein).
\end{example}

\begin{example}
For an exponent vector $b \in \N^n$ with $b_i \ge 2$ for all $i$, the irreducible ideal $\m^b$ has a unique socle element $x^{b-1} := x_1^{b_1-1} \ldots x_n^{b_n - 1}$. Let $M := \m^b + (x^{b-1})$, which is Artinian level with $n$-dimensional socle $\langle \frac{x^{b-1}}{x_i} \mid 1 \le i \le n \rangle =: \langle s_1, \ldots, s_n \rangle$. By setting all socle elements of $M$ equal to each other we obtain a graded ideal $I := M + (s_1 - s_i \mid 2 \le i \le n)$. As all the non-monomial generators are in the socle of $R/M$, we may apply \Cref{socTerms}: the coefficient matrix $(a_{ij})$ is given by
\[
\begin{bmatrix*}[r]
1 & 1 & \ldots & 1 \; \\
-1 & 0 & \ldots & 0 \; \\
0 & -1 & \ldots & \vdots \,\; \\
\vdots & & \ddots & 0 \; \\
0 & \ldots & 0 & -1 \;
\end{bmatrix*}
\]
so by \Cref{socTerms}, $\mono(I) = M$. For $b = (2,2,3,3)$, the Betti tables of $R/I$ and $R/\mono(I)$ respectively are:

\begin{verbatim}
               0 1  2  3 4                 0 1  2  3 4
        total: 1 7 17 16 5          total: 1 5 10 10 4
            0: 1 .  .  . .              0: 1 .  .  . .
            1: . 2  .  . .              1: . 2  .  . .
            2: . 2  1  . .              2: . 2  1  . .
            3: . .  4  . .              3: . .  4  . .
            4: . 3 12 16 4              4: . .  1  2 .
            5: . .  .  . 1              5: . 1  4  8 4
\end{verbatim}

Notice that the (total) Betti numbers of $R/I$ are strictly greater than those of $R/\mono(I)$. This shows that the answer to \Cref{Q7} is false in general.
\end{example}

Finally, we include a criterion for recognizing when a monomial subideal of $I$ is equal to $\mono(I)$, in terms of its socle monomials:

\begin{proposition}
Let $I$ be an $R$-ideal and $M \subseteq I$ an Artinian monomial ideal. Then the following are equivalent:

\begin{enumerate}[label={\upshape(\alph*)}]
\item $M = \mono(I)$

\item $I$ contains no socle monomials of $M$

\item $(M : \m) \cap \mono(I) \subseteq M$
\end{enumerate}

\end{proposition}

\begin{proof}
(a) $\implies$ (b), (a) $\implies$ (c): Clear.

(b) $\implies$ (a): Recall that for an Artinian monomial ideal $M$, a monomial $x^b$ is nonzero in $\soc(R/M)$ iff $\m^{b+1}$ appears in the unique irreducible decomposition of $M$ into irreducible monomial ideals (here $b+1 = (b_1 + 1, \ldots, b_n + 1)$: cf. \cite{MS}, Exercise 5.8). Let $x^{b_1}, \ldots, x^{b_r}$ be the socle monomials of $R/M$. Then by assumption $x^{b_1}, \ldots, x^{b_r}$ are also socle monomials of $R/\mono(I)$, so $\m^{b_1 + 1}, \ldots, \m^{b_r + 1}$ all appear in the irreducible decomposition of $\mono(I)$, hence $\mono(I) \subseteq M$.

(b) $\iff$ (c): Notice that (b) is equivalent to: any monomial $u \in (M : \m) \setminus M$ is not in $I$; or alternatively, any monomial in both $M : \m$ and $I$ is also in $M$; i.e. $\mono((M : \m) \cap I) \subseteq M$. Now apply \Cref{basicProp}(3). 
\end{proof}

\end{document}